\documentclass[11pt,draft]{article}
\usepackage{amsfonts,amsmath,amssymb,amstext,amsbsy,amsopn,amsthm}

\newcommand{\ee}{{\rm e}}
\newcommand{\bda}{\mathbf{a}}
\newcommand{\bdb}{\mathbf{b}}
\newcommand{\beq}{\begin{eqnarray}}
\newcommand{\eeq}{\end{eqnarray}}
\newcommand{\beqn}{\begin{eqnarray*}}
\newcommand{\eeqn}{\end{eqnarray*}}

\newcommand{\itg}{\int \limits}

\def\C{{\mathbb C}}
\newcommand{\Ree}{\mbox{\rm Re}\:}

\newcommand{\cD}{\mathcal D}

\newcommand{\cK}{\mathcal K}

\newcommand{\cM}{\mathcal M}

\newcommand{\cO}{\mathcal O}
\newcommand{\cP}{\mathcal P}
\newcommand{\cQ}{\mathcal Q}

\newcommand{\cF}{\mathcal F}

\newcommand{\bdx}{\mathbf{x}}
\newcommand{\bdy}{\mathbf{y}}
\newcommand{\bdm}{\mathbf{m}}
\newcommand{\bdh}{\mathbf{h}}
\newcommand{\bdk}{\mathbf{k}}

\DeclareMathOperator{\erfc}{erfc}

\DeclareMathOperator{\e}{e}

\def\R{{\mathbb R}}
\def\Z{{\mathbb Z}}

\newcommand\ds{\displaystyle}
\def\l{\lambda}

\def\b{{\beta}}
\def\a{{\alpha}}

\newtheorem{thm}{Theorem}[section]
\newtheorem{rmk}{Remark}[section]
\newcommand{\subjclass}{{\bf  Mathematics Subject Classification (2000). }}

\title{Fast cubature of volume potentials over rectangular domains}

\author{{  F. Lanzara$^{\mbox{\tiny 1}}$ , V. Maz'ya$^{\mbox{\tiny 2}}$ ,
G. Schmidt$^{\mbox{\tiny 3}}$}
}
\date{}

\begin{document}
\maketitle

\hspace*{1cm}
\parbox{10cm}{\begin{flushleft}
{\footnotesize\em
\begin{itemize}
\item[$^{\mbox{\tiny\rm 1}}$] Department of Mathematics, Sapienza University of Rome,\\
Piazzale Aldo Moro 2, 00185 Rome, Italy\\
\texttt{\rm lanzara\symbol{'100}mat.uniroma1.it}
\item[$^{\mbox{\tiny\rm 2}}$]Department of Mathematics, University of
Link\"oping, \\ 581 83 Link\"oping, Sweden;\\
Department of Mathematical Sciences, M\&O Building, University of
Liverpool, Liverpool L69 3BX, UK;\\
\texttt{\rm vlmaz\symbol{'100}mai.liu.se }
\item[$^{\mbox{\tiny\rm 3}}$]Weierstrass Institute for Applied Analysis and
Stochastics, \\  Mohrenstr. 39,
10117 Berlin, Germany \\
\texttt{\rm schmidt\symbol{'100}wias-berlin.de}
\end{itemize}
}
\end{flushleft}}

\begin{abstract}
In the present paper we study high-order cubature formulas for the computation of 
advection-diffusion potentials over boxes. By using the basis functions introduced 
in the theory of approximate approximations, the cubature of a potential is reduced 
to the quadrature of one dimensional integrals. For densities with separated approximation,
we derive a tensor product representation of the integral operator which
admits efficient cubature procedures in very high dimensions.
Numerical tests show that these formulas are accurate and provide 
approximation of order $\cO(h^6)$ up to dimension $10^8$.
\end{abstract}


\subjclass 65D32; 65-05; 41A30; 41A63.

\section{Introduction}
\setcounter{equation}{0}

High-dimensional volume potentials arise in many mathematical models  in the field of physics, chemistry, biology, financial mathematics and many others. In recent years, tensor product approximation has been recognized as a successful tool to overcome the "curse of dimensionality"  and treat high-dimensional integral operators as described, for example, in \cite{HK,HK2,Khor,BHK}.

In the present paper we propose to combine  high-order semi-analytic cubature formulas, obtained by using the method
of approximate approximations (see \cite {MSbook} and the reference therein), with tensor product approximations.  

Cubature formulas based on approximate approximations for volume potentials over $\R^n$ and over bounded domains have been considered in \cite{MS2} and \cite{LMS3}, respectively (see also \cite{MSbook}).  The cubature of  high-dimensional volume potentials over the full space and over half-spaces  has  been studied in \cite{LMS2} and \cite{LMS}. 
Now we consider the volume potential 
\begin{equation}\label{vol}
\cK_\l f (\bdx)=\int_{[P,Q]}\kappa_{\l}(\bdx-\bdy)f(\bdy)d\bdy,
\end{equation}
with the fundamental solution 
\begin{align*}
\kappa_\l(\bdx)=
\frac{1}{(2 \pi)^{n/2}}
\Big( \frac{|\bdx|}{\lambda}\Big)^{1-n/2}K_{n/2-1} (\lambda|\bdx|) \, ,\l\in\C\setminus(-\infty,0],
\end{align*}
over rectangular domains $[P,Q]=\prod_{j=1}^n [P_j,Q_j]\subset \R^n$.
Here $K_\nu$ is the modified Bessel
function of the second kind (see \cite[9.6, p.374]{AS}).

The function $u=\cK f$ provides a solution of the modified Helmholtz equation
\[
(-\Delta+\l^2) u=\left\{ \begin{array}{cc} f(\bdx)& \bdx\in [P,Q]\\ \\ 0& otherwise.
\end{array}\right.
\]
For $\l=0$, then
\begin{align*}
\kappa_0(\bdx)
=\left\{
\begin{array}{ll}
\ds \frac{1}{2\pi} \log \frac{1}{|\bdx|} \, , &n=2\, ,\\
\ds \frac{\Gamma(\frac{n}{2}-1)}
    {4 \pi^{n/2}} \frac{1}{|\bdx|^{n-2}}\, ,
&n\geq 3\, 
\end{array}
\right.
\end{align*}
is the fundamental solution of the Laplacian.

The theory of  approximate approximations proposes semi-analytic cubature formulas
for volume potentials by using
quasi-interpolation of the density $f$ by
functions for which the integral operator can be taken analytically.
Approximate quasi-interpolant has the form
\begin{equation*}\label{quasiint}
\cM_{h,\cD}{f}(\bdx)=\cD^{-n/2}
\sum_{\bdm\in\Z^n}{f}(h\bdm)\eta\left(\frac{\bdx-h\bdm}
{h\sqrt{\cD}}\right)
\end{equation*}
where $h$ and $\cD$ are positive parameters and  $\eta$ is a smooth and rapidly decaying function which satisfies the moment conditions of order $N$
\begin{equation}\label{moment}
\int_{\R^n}\eta(\bdx)\,\bdx^{\alpha}d\bdx=\delta_{0,\alpha},\quad 0\leq
|\alpha|<N.
\end{equation}

If $f\in C^N_0(\R^n)$,  it is known (\cite{MSbook}) that
\begin{equation*}
|{f}(\bdx)-\cM_{h,\cD}{f}(\bdx)| \le
c (\sqrt{\cD}h)^N \|\nabla_N f\|_{L^\infty}
+ \sum_{k=0}^{N-1} \varepsilon_k (\sqrt{\cD}h)^k \big|\nabla_k f(\bdx)\big|
\end{equation*}
with
\begin{align*}
\varepsilon_k \le\sum_{\bdm \in {\Z}^n \setminus \{0\}}
  \big| \nabla_k \cF\eta (\sqrt{\cD} \bdm)  \big| \, ; \lim_{\cD\to \infty}
\sum_{\bdm \in {\Z}^n \setminus \{0\}}
  \big| \nabla_k \cF\eta (\sqrt{\cD} \bdm)\big|=0.
\end{align*}
If we replace $f$ in \eqref{vol} by the quasi-interpolant 
\begin{equation}\label{quasiint2}
\cD^{-n/2}
\sum_{h\,\bdm\in [P,Q]}{f}(h\bdm)\eta\left(\frac{\bdx-h\bdm}
{h\sqrt{\cD}}\right)
\end{equation}
we don't obtain good approximations because  \eqref{quasiint2} approximates $f$ only in a subdomain of $[P,Q]$ with positive distance from the boundary.  To avoid this difficulty 
we extend $f$  with preserved
smoothness in a larger domain. Obviously  the quasi-interpolant of the continuation
  $\widetilde{f}$ approximates $f$ in $[P,Q]$.  Assume that
there exists $C>0$ such that
\[
||\widetilde{f}||_{W^N_\infty(\R^n)}\leq C\, ||f||_{W^N_\infty([P,Q])}.
\]
Since $\eta$ is a smooth and rapidly decaying function,
for any error $\epsilon >0$ one can fix $r>0$ and the parameter $\cD>0$ such
that the quasi-interpolant with nodes in a neighborhood of $[P,Q]$
\begin{equation*}\label{quasiintr} 
\cM^r_{h,\cD}\widetilde{f}(\bdx)=
\cD^{-n/2}\sum_{d(h\bdm,[P,Q])\leq
r\,h\sqrt{\cD}}
\widetilde{f}(h\bdm)\eta\left(\frac{\bdx-h\bdm}
{h\sqrt{\cD}}\right)
\end{equation*}
approximates $f$ with
\begin{equation}\label{estimate}
|{f}(\bdx)-\cM^r_{h,\cD}\widetilde{f}(\bdx)| = \cO((\sqrt{\cD}h)^N+ \epsilon)
\|f\|_{W^N_\infty}
\end{equation}
for all $\bdx \in [P,Q]$.

Then the integral 
\begin{equation*}\label{cub}
\cK_{\l,h}\widetilde{f}(\bdx)=\cK_\l(\cM^r_{h,\cD}\widetilde{f}) (\bdx)
=
\cD^{-n/2} \!\!\!\!\! \!\!\!\!\!\sum_{d(h\bdm,[P,Q])\leq
r\,h\sqrt{\cD}}\!\!\!\!\!
\widetilde{f}(h\bdm) \int_{[P,Q]}
\kappa_{\l}(\bdx-\bdy)\,
\eta\left(\!\frac{\bdy-h\bdm}{h\sqrt{\cD}}\right)d\bdy
\end{equation*}
 gives a cubature of \eqref{vol}.
 
Since $\cK_\l$ is a bounded mapping between suitable function spaces, the
differences
$\cK_{\l,h}\widetilde{f}(\bdx)-\cK_{\l}f(\bdx)$ behave like estimate
\eqref{estimate}.
Therefore,  to construct  high order cubature formulas for \eqref{vol}, it
remains to compute  the integrals
\[
  \int_{[P,Q]}  \kappa_{\l}(\frac{\bdx-h\bdm}{h\sqrt{\cD}}-\bdy)\,
\eta(\bdy)d\bdy
\]
for nodes with $d(h\bdm,[P,Q])\leq
r\,h\sqrt{\cD}$.
This is performed by using one-dimensional integral representations.
As basis functions we take the tensor products of univariate basis functions
\begin{align}\label{tensorbasis}
\widetilde{\eta}(\bdx)=\prod_{j=1}^{2M} \widetilde \eta_{2M}(x_j); \quad \widetilde \eta_{2M}(x_j) =
\pi^{-1/2} L_{M-1}^{(1/2)}(x_j^2) \e^{\, - x_j^2} 
\end{align}
which satisfies the moment condition \eqref{moment} of order $N=2M$ (cf. \cite{MSbook}), where  $L_{k}^{(\gamma)}$ are the generalized Laguerre polynomials
     \begin{equation*} \label{defLag}
L_{k}^{(\gamma)}(y)=\frac{\e^{\,y} y^{-\gamma}}{k!} \, \Big(
{\frac{d}{dy}}\Big)^{k} \!
\left(\e^{\,-y} y^{k+\gamma}\right), \quad \gamma > -{\rm 1} \, .
     \end{equation*}
     
Using the representation with a tensor product integrand
\begin{equation*}
\int_{[P,Q]} \cK_\l(\bdx-\bdy) \ee^{-|\bdy|^2}d \bdy=
\end{equation*}
\begin{equation}\label{1dint}
\frac{1}{4} \itg_0^\infty \e^{-   \l^2 t/4}  \prod_{j=1}^n
\frac{\e^{-x_j^2/(1+ t)}}{2 \sqrt{\pi}}
\Big(
\erfc \big(\sqrt{\frac{1+t}{t}} \Big( P_j -
  \frac{ x_j}{1+t} \Big) \big)-\erfc\big(\sqrt{\frac{1+t}{t}} \Big( Q_j-
  \frac{ x_j}{1+t}\Big)\big)
  \Big) dt
\end{equation}
we derive  a tensor product representation of the integral operator which admits efficient cubature procedures for densities with separated approximation (Section \ref{secone}).  We will consider quasi-interpolants  \eqref{quasiint_ani} on anisotropic grids which use different step size $h_j>0,j=1,...,n$ along different space dimensions.  If $h_j= \tau h$ , $0<\tau\leq 1$ the error of the quasi-interpolant  \eqref{quasiint_ani}  is always $\cO(h^N)$.
  In Section \ref{numres2} we provide numerical tests, showing that these formulas are accurate and provide approximation of order $\cO(h^6)$ up to dimension $10^8$.

\section{Higher order cubature formula based on \eqref{1dint}}\label{secone}
\setcounter{equation}{0}

In this section we describe a  high order cubature of $\cK_\l f$
  in the case of rectangular domain in $\R^n$. Let
\[
[P,Q]= \{\bdx=(x_1,\ldots,x_n): P_j\leq x_j \leq Q_j, j=1,...,n\}=\prod_{j=1}^n[P_j,Q_j] \, .
\]
As basis functions we use  \eqref{tensorbasis}.

In order to apply also quasi-interpolants
on rectangular grids $(h_1 m_1,\ldots,h_n m_n), h_j > 0$, shortly
denoted by $\{\bdh \bdm\}$,
\begin{equation}\label{quasiint_ani}
\cM_{h,\cD}\widetilde{f}(\bdx)=\cD^{-n/2}
\sum_{\bdm\in\Z^n}\widetilde{f}(\bdh \bdm)
\prod_{j=1}^n \widetilde\eta_{2M} 
\left(\frac{x_j-h_jm_j}{h_j\sqrt{\cD}}\right) \, ,
\end{equation}
we define the basis function
$\eta(\bdx)=\prod \widetilde \eta_{2M}(a_j x_j)$, $a_j > 0$,
and look for integral representations
of  the solution of
\begin{equation}\label{advecM}
(-\Delta
+\l^2) \, u = \prod_{j=1}^n \chi_{(p_j,q_j)}(x_j) \, \widetilde
\eta_{2M}(a_j x_j)\, .
\end{equation}
Here $\chi_{(p_j,q_j)}$ is the characteristic function of
the interval $(p_j,q_j)$ with
$-\infty \le p_j < q_j \le +\infty$,
$j=1,\ldots,n$. 
\begin{thm} \label{lem1} Let $\Ree \l^2\ge 0$ and $n \ge 3$.
The solution of
equation \eqref{advecM} in $\R^n$
can be expressed by the one-dimensional integral
\begin{align} \label{zwint}
u(\bdx)
=
\frac{1}{4} \itg_0^\infty \e^{-   \l^2 t/4}  \prod_{j=1}^n
\Big(\Phi_M(a_jx_j,a_j^2    t,a_jp_j) -\Phi_M(a_jx_j,a_j^2 
t,a_jq_j)\Big)  \, dt
\end{align}
where the function $\Phi_M$ is given by
\begin{equation*} \label{defPhiM}
\Phi_M(x,t,p)=\frac{\e^{-x^2/(1+ t)}}{2 \sqrt{\pi}}\Big(
{\erfc\big(F(t,x,p)\big)} \cP_M(t,x)- \frac{\e^{-F^2(t,x,p)}}{\sqrt{\pi}}
\cQ_M\big(t,x,p\big)\Big)
\end{equation*}
with the function
\begin{equation*} \label{defF}
F(t, x, y)  =
  \sqrt{\frac{1+t}{t}} \Big( y -
  \frac{ x}{1+t} \Big),
\end{equation*}
and $\cP_M,  \cQ_M$  are polynomials in $x$
of degree $2M-2$ and $2M-3$, respectively:
\begin{align*}
\cP_M(t,x)&=
\sum_{k=0}^{M-1} \frac{1}{(1+ t)^{k+1/2}}
L^{(-1/2)}_{k}\Big(\frac{x^2}{1+t}\Big) , \\
\cQ_M(t,x,y) &=2
\sum_{k=1}^{M-1} \frac{(-1)^k}   {k! \, 4^k}
  \sum_{\ell=1}^{2k}
\frac{(-1)^{\ell}}{t^{\ell/2}}
\bigg( H_{2k-\ell} (y)
  H_{\ell-1} \Big( \frac{y-x}{\sqrt{t}}\Big) \\
&\hskip60pt -
\Big(\hskip-4pt\begin{array}{c}2k\\\ell\end{array}\hskip-4pt\Big)
H_{2k-\ell}\Big(\frac{x}{\sqrt{1+t}}\Big)
\frac{H_{\ell-1}\big(F(t, x, y)\big)}{(1+t)^{k+1/2}}
\bigg).
\end{align*}
If $\Ree   \l^2 > 0$, then the representation
\eqref{zwint} is valid for all $n \ge 1$.

By $H_k$
we denote the Hermite polynomials
\begin{align} \label{defHerm}
H_k(x)=(-1)^k \e^{x^2}
\frac{d^k}{dx^k}
\e^{-x^2} \, .
\end{align}
\end{thm}
\begin{proof}
The solution of \eqref{advecM}
can be obtained  explicitly by
  using the parabolic equation
\begin{align}
&\partial_t w -   \Delta   w +  \l^2 w =0 \, , \quad t \ge 0 \, ,\label{theat}
\intertext{ with the initial condition}
&w(\bdx ,0) =
\prod_{j=1}^n \chi_{(p_j,q_j)}(x_j) \, \widetilde \eta_{2M}(a_j x_j) 
\, . \nonumber
\end{align}
Integrating \eqref{theat} in $t$ we derive
\begin{align*}
w(\bdx,T)-w(\bdx,0)-  (\Delta- \l^2)\itg_0^Tw(\bdx,t) \, dt =0 \, ,
\end{align*}
hence  the solution of \eqref{advecM} is
expressed as the one-dimensional integral
\begin{align*}
u(\bdx) = \itg_0^\infty w(\bdx,t) \, dt
\end{align*}
provided
it exists.
Obviously, if $w$ solves \eqref{theat},
then $z=w \e^{\,   \l^2 t}$ is the solution of
the initial value problem  for the heat equation
\begin{equation*}\label{theat1}
\partial_t z -   \Delta   z  =0 \, , \quad  z(\bdx ,0) =
\prod_{j=1}^n \chi_{(p_j,q_j)}(x_j) \, \widetilde \eta_{2M}(a_j x_j) \, ,
\end{equation*}
which has, by Poisson's formula, the solution
\begin{align*}
z(\bdx,t) & =\frac{1}{({4 \pi   t)^{n/2} }} \itg_{\prod (p_j,q_j)}
  {\e}^{-|\bdx-\bdy|^2/(4   t)}
\prod_{j=1}^n
\widetilde \eta_{2M}(a_j y_j)
\, d\bdy \, \\
& =  \prod_{j=1}^n \frac{1}{\pi^{1/2} (4  a_j^2    t)^{1/2}}\itg_{a_j 
p_j}^{a_j q_j}
\e^{-(a_j x_j- y_j)^2/(4 a_j^2    t)}\widetilde \eta_{2M}(y_j)
\, dy_j
\end{align*}
where $\prod (p_j,q_j)$ is the Cartesian product of the intervals
$(p_j,q_j)$.
Denoting
\begin{align*}
\Phi_M(x,t,p)=\frac{1}{\sqrt{\pi t}}\itg_{p}^{\infty}
\e^{-(x- y)^2/t}\widetilde \eta_{2M}(y)
\, dy
\end{align*}
we get the one-dimensional integral representation \eqref{zwint}
of the solution of \eqref{advecM},
provided  this integral exists.
Denoting
\begin{equation*} \label{defvarphi}
\varphi_k(x, t,p) = \itg_{p}^{\infty} {\e}^{-(x- y)^2/t}
  \, \frac{d^{2k}}{dy^{2k}}  \e^{-y^2} \, dy
   \end{equation*}
and using the general representation \cite[p.55]{MSbook}
\begin{equation*}\label{gen}
\eta_{2M}(\bdx)=\pi^{-n/2} \sum_{j=0}^{M-1} \frac{(-1)^j}{j!4^j}\Delta^j
\e^{-{|\bdx|^2}},
\end{equation*}
we have
\begin{align*}
\Phi_M(x,t,p)
=  \frac{1}{\pi \sqrt{t}}\sum_{k=0}^{M-1} \frac{(-1)^k}   {k! \, 4^k} \,
\varphi_k(x, t,p) \, .
\end{align*}
From
\begin{align*}
\varphi_{0}(x,t,p)=\itg_{p}^{\infty} \e^{-(x-y)^2/t} \e^{-y^2}    \,
dy=\frac{\sqrt{\pi}}{2}\sqrt{\frac{t}{1+t}}\e^{-x^2/(1+t)}
\erfc\big(F(t,x,p)\big) \, ,
\end{align*}
for $k \geq 1$, integration by parts   leads to
\begin{align*}
\varphi_k(x, t,p)=\frac{\partial^{2k}\varphi_{0}(x,t,p)}{\partial x^{2k}}
-\sum_{\ell=0}^{2k-1}(-1)^\ell
\frac{\partial^{\ell}}{\partial y^{\ell}}
\e^{-(x - y)^2/t}
\frac{d^{2k-\ell-1}}{d y^{2k-\ell-1}}\e^{-y^{2}}
\Bigg|_{y=p} 
\end{align*}
and the definition \eqref{defHerm}
gives
\begin{align*}
\frac{d^{2k-\ell-1}}{d y^{2k-\ell-1}}\e^{-y^{2}}
&=(-1)^{2k-\ell-1}
\e^{-y^{2}} H_{2k-\ell-1}(y) \, ,\\
 \frac{\partial^{\ell}}{\partial y^{\ell}}
\e^{-(x - y)^2/t}
&=\frac{(-1)^\ell\e^{-(x - y)^2/t}}{t^{\ell/2}} \,
H_\ell \Big( \frac{y-x}{\sqrt{t}}\Big) .
\end{align*}
In view of
\[
  \frac{d^{\ell}}{dx^{\ell}}\erfc(x)= \frac{2}{\sqrt{\pi}}(-1)^{\ell}
\e^{-x^2} H_{\ell-1}(x)
,\quad \ell \geq 1 \, ,
\]
one gets for $\ell<2k$
\begin{align*}
\frac{\partial^{2k-\ell}}{\partial x^{2k-\ell}} \erfc\big(F(t,x,p)\big)&=
\frac{(-1)^{2k-\ell}}{(t(1+t))^{k-\ell/2}}
\left[ \frac{d^{2k-\ell}}{dz^{2k-\ell}}\erfc(z)\right]_{z=F(t,x,p)}\\
&=\frac{2 \e^{-F^2(t,x,p)} }{\sqrt{\pi}(t(1+ t))^{k-\ell/2}} \,
H_{2k-\ell-1}(F(t,x,p)) \,.
\end{align*}
Therefore,  since
\[
\frac{d^\ell}{d x^\ell} \e^{-x^2/(1+t)}
=\frac{(-1)^\ell\e^{-x^2/(1+t)}}{(1+t)^{\ell/2}}
  H_\ell\Big(\frac{x}{\sqrt{1+t}}\Big)\, ,
\]
we obtain
\begin{align*}
\frac{\partial^{2k}}{\partial x^{2k}}& \, \varphi_{0}(x,t,p)
=  \frac{\sqrt{\pi t}}{2}\frac{\e^{-x^2/(1+t)}}{(1+ t)^{k+1/2}}
H_{2k}\Big(\frac{x}{\sqrt{1+t}}\Big)\erfc(F(t,x,p))\\
&\; -\frac{ \sqrt{t} \e^{-x^2/(1+t)} \e^{-F^2(t,x,p)}}{(1+ t)^{k+1/2}}
  \sum_{\ell=0}^{2k-1}
\Big(\hskip-4pt\begin{array}{c}2k\\\ell\end{array}\hskip-4pt\Big)
\frac{(-1)^{\ell}}{t^{k-\ell/2}}
H_\ell\Big(\frac{x}{\sqrt{1+t}}\Big)
H_{2k-\ell-1}(F(t,x,p)) \, .
\end{align*}
Thus simple transformations give
\begin{align}
 \varphi_{k}(x,t,p)=& \e^{-x^2/(1+t)} \bigg(\erfc\big(F(t,x,p)\big)
H_{2k}\Big(\frac{x}{\sqrt{1+t}}\Big)\frac{\sqrt{\pi\, t}}{2(1+t)^{k+1/2}}\nonumber
\\
&+\e^{-F^2(t,x,p)}
  \sum_{\ell=1}^{2k}\frac{(-1)^{\ell}}{t^{(\ell-1)/2}}\nonumber \\
&\quad \times\bigg( \! \Big(\hskip-3pt\begin{array}{c}2k\\\ell\end{array}\hskip-3pt\Big)
H_{2k-\ell}\Big(\frac{x}{\sqrt{1+t}}\Big)
\frac{H_{\ell-1}(F(t,x,p))}{(1+t)^{k+1/2}}
-H_{\ell-1}\Big(\frac{p-x}{\sqrt{t}}\Big)
H_{2k-\ell}(p) \!\bigg)\bigg).\nonumber
\end{align}
Using the relation
$H_{2k}(x)=(-1)^k \, 4^k\, k!\, L_k^{(-1/2)}(x^2)$
we find therefore
\begin{align*}
\Phi_M(x,t,p)&=\frac{\e^{-x^2/(1+t)} \erfc\big(F(t,x,p)\big)}{2 \sqrt{\pi}}
\sum_{k=0}^{M-1} \frac{1}{(1+t)^{k+1/2}}
L^{(-1/2)}_{k}\Big(\frac{x^2}{1+t}\Big)\\
&\quad 
+\frac{\e^{-x^2/(1+t)}\e^{-F^2(t,x,p)}}{\pi}\sum_{k=0}^{M-1}\frac{(-1)^k} 
{k! \, 4^k}
  \sum_{\ell=1}^{2k}\frac{(-1)^{\ell}}{t^{\ell/2}} \\
&\quad \times
\bigg( \! \Big(\hskip-3pt\begin{array}{c}2k\\\ell\end{array}\hskip-3pt\Big)
H_{2k-\ell}\Big(\frac{x}{\sqrt{1+t}}\Big)
\frac{H_{\ell-1}\big(F(t,x,p)\big)}{(1+t)^{k+1/2}}
-H_{\ell-1}\Big(\frac{p-x}{\sqrt{t}}\Big)
H_{2k-\ell}(p) \!\bigg) \\
  &=\frac{\e^{-x^2/(1+ t)}}{2 \sqrt{\pi}}\Big(
  \erfc\big(F(t,x_,p)\big) \cP_M(t,x)- \frac{\e^{-F^2(t,x,p)}}{\sqrt{\pi}}
  \cQ_M\big(t,x,p\big)\Big) .
\end{align*}
\end{proof}

The polynomials $\cP_M(t,x)$ and $\cQ_M(t,x,p)$
for $M=1,2,3$ are given by
\begin{align*}
&\cP_1(t,x)=\frac{1}{(1+t)^{1/2}} \, , \qquad
\cP_2(t,x)=\cP_1(t,x)+\frac{1}{2(1+t)^{3/2}}-\frac{x^2}{(1+t)^{5/2}}\, ,\\
&\cP_3(t,x)=\cP_2(t,x)+\frac{3}{8(1+t)^{5/2}}-
\frac{3\, x^2}{2(1+t)^{7/2}}+\frac{x^4}{2(1+t)^{9/2}}\, ,\\
&\cQ_1(t,x,p)=0,\qquad \cQ_2(t,x,p)=\frac{\sqrt{t}}{(1+t)} 
\Big(\frac{x}{1+t} +p\Big),
   \\
&\cQ_3(t,x,p)=-\frac{\sqrt{t}}{4(1+t)} \Big(
\frac{2 x^3}{(1+t)^3}+\frac{2 px^2 -5x}{(1+t)^2}
+\frac{(2p^2 -5)x-3p}{1+t}+p(2p^2-7)
\Big).
\end{align*}
\begin{rmk} \label{remark1}
Since for positive $r$
\[
0 < \erfc (r) \le \e^{-r^2}
  \quad \mbox{and}
\quad 2- \e^{-r^2} < \erfc (-r) < 2
\]
from the relation
\[
F^2(t,x,p) = p^2+\frac{(x-p)^2}{t} - \frac{x^2}{1+t}
\]
we get
\begin{align*}
|\e^{-x^2/(1+ t)}\erfc\big(F(t,x,p)\big)|\le \e^{-p^2} \quad \mbox{if}
\quad p > 0
  \end{align*}
and
\begin{align*}
|\e^{-x^2/(1+ t)}\erfc\big(F(t,x,p)\big)-2\e^{-x^2/(1+ t)} | < 
\e^{-p^2}   \quad \mbox{if}
\quad p <0 \, .
  \end{align*}
Thus for sufficiently large $|p|$
\begin{align*}
\Phi_M(x,t,p) = \left\{
\begin{array}{cc}
\pi^{-1/2} \e^{-x^2/(1+ t)}
\cP_M(t,x) + \cO(\e^{-p^2}) \, & \mbox{if } p  < 0 \, ,\\
\cO(\e^{-p^2}) \, & \mbox{if } p  > 0\, ,
\end{array} \right.
\end{align*}
and therefore, for sufficiently large $r$
one can use the approximation
\begin{align*}
\Phi_M(x,t,p) -\Phi_M(x,t,q)
\approx
\left\{
\begin{array}{cc}
0 \, ,& p,q  \ge r \mbox{ or } p,q  \le -r \, ,\\
\pi^{-1/2}\e^{-x^2/(1+ t)}
\cP_M(t,x) , &p  \le -r \mbox{ and } q  \ge r \, ,
\end{array} \right.
\end{align*}
with the error $\cO(\e^{-r^2})$.
Similarly, if $q-p \ge 2r$, then
\begin{align*}
\Phi_M(x,t,p) -\Phi_M(x,t,q)
\approx
\left\{
\begin{array}{cc}
\Phi_M(x,t,p) \, ,& -r < p  < r \, ,\\
\pi^{-1/2}\ \e^{-x^2/(1+ t)}
\cP_M(t,x)-\Phi_M(x,t,q) , &   -r < q  < r  \, .
\end{array} \right.
\end{align*}

\end{rmk}

\section{Implementation and numerical results} \label{numres2}
\setcounter{equation}{0}

We compute the  cubature formula
\begin{align*}
\cK_{\l,\bdh}\widetilde{f}(\bdx)=
\cD^{-n/2}
\sum_{\bdh\bdm\in\widetilde\Omega_{r \bdh}}\widetilde{f}(\bdh \bdm)
\itg_{{[P,Q]}}\kappa_\l(\bdx-\bdy)
\prod_{j=1}^n \widetilde\eta_{2M} 
\left(\frac{y_j-h_jm_j}{h_j\sqrt{\cD}}\right) \, d \bdy
\end{align*}
where $\widetilde\Omega_{r \bdh}=\prod_{j=1}^n(P_j-rh_j\sqrt{\cD},Q_j+rh_j\sqrt{\cD})$,
using   the tensor product representation of Theorem  \ref{lem1}.
At the grid points $\bdh \bdk=(h_1 k_1,\ldots,h_n k_n)$
we obtain
\begin{align*}
&\itg_{[P,Q]}\kappa_\l(\bdh \bdk-\bdy)
\prod_{j=1}^n \widetilde\eta_{2M} 
\left(\frac{y_j-h_jm_j}{h_j\sqrt{\cD}}\right) \, d \bdy 
=
\frac{1}{4}  \itg_0^\infty \e^{-   \l^2 t/4}  \\
&\quad \times \prod_{j=1}^n
\Big(\Phi_M\big(\frac{k_j-m_j}{\sqrt{\cD}},
\frac{   t}{h_j^2 \cD},
\frac{P_j-h_jm_j}{h_j\sqrt{\cD}}\big) -
\Phi_M\big(\frac{k_j-m_j}{\sqrt{\cD}},
\frac{   t}{h_j^2 \cD},
\frac{Q_j-h_jm_j}{h_j\sqrt{\cD}}\big)\Big)  \, dt
\end{align*}
and therefore
\begin{align} \label{defquadbox}
\cK_{\l,\bdh}\widetilde{f}(\bdx)=
\sum_{\bdh\bdm\in\widetilde\Omega_{r \bdh}}\widetilde{f}(\bdh \bdm)
\, \bdb_{\bdk,\bdm}^{(M)} \, ,
\end{align}
where we introduce the one-dimensional integral
\begin{align} \label{defbdb}
\bdb^{(M)}_{\bdk,\bdm} = \frac{1}{4 \cD^{n/2}}  \itg_0^\infty \e^{-   \l^2 t/4}
\prod_{j=1}^n \Big(b^j_{k_j,m_j}(P_j)-b^j_{k_j,m_j}(Q_j)\Big) \, dt
\end{align}
and use the abbreviation
\begin{align*}
b^j_{k,m}(P)=\e^{-(k-m)^2/(\cD(1+ t))}\bigg(
\erfc\big(F(\frac{t}{h_j^2 
\cD},\frac{k-m}{\sqrt{\cD}},\frac{P-h_jm}{h_j\sqrt{\cD}})
\big) \cP_M\big(\frac{t}{h_j^2 \cD},\frac{k-m}{\sqrt{\cD}}\big)\\
 - \pi^{-1/2}\exp\Big(-F^2\big(\frac{t}{h_j^2 \cD},
\frac{k-m}{\sqrt{\cD}},\frac{P-h_jm}{h_j\sqrt{\cD}}\big)\Big)
\cQ_M\big(\frac{t}{h_j^2 
\cD},\frac{k-m}{\sqrt{\cD}},\frac{P-h_jm}{h_j\sqrt{\cD}}
\Big)\bigg)/(2\sqrt{\pi}).
\end{align*}
According to Remark \ref{remark1}, for appropriately chosen $r> 0$
we can set within a given accuracy
\begin{align*}
&b^j_{k,m}(P) = a_{k-m}^j
= \pi^{-1/2}\e^{-(k-m)^2/( \cD(1+   t))}
\cP_M\big(\frac{t}{h_j^2 \cD},\frac{k-m}{\sqrt{\cD}}\big) &\mbox{if} \;
P-h_jm \le  -r h_j \sqrt{\cD} \, , \\
&b^j_{k,m}(P)=0 &\mbox{if} \hskip5pt   P-h_jm \ge r h_j \sqrt{\cD}\, 
,\hskip5pt
\end{align*}
which speeds up the computation of \eqref{defbdb}.
In particular,
we can split \eqref{defquadbox}
into
\begin{equation}\label{defquadbox2}
\cK_{\l,\bdh}^{(M)} f(\bdh\bdk)=
\sum_{\bdh\bdm \in\Omega_{r \bdh}}
\hspace{-5pt}{f}(\bdh\bdm)\,\bda_{\bdk-\bdm}^{(M)} +
\sum_{\bdh\bdm \in
\widetilde\Omega_{r \bdh}\setminus \Omega_{r \bdh} }
\hspace{-5pt}\widetilde{f}(\bdh\bdm)\bdb_{\bdk,\bdm}^{(M)}\, ,
\end{equation}
where
$\Omega_{r \bdh}=\prod_{j=1}^n(P_j+rh_j\sqrt{\cD},Q_j-rh_j\sqrt{\cD})$,
and the coefficients in the convolutional sum are given by
\begin{align*}
\bda^{(M)}_{\bdk} &= \frac{1}{4 \cD^{n/2}}  \itg_0^\infty \e^{-   \l^2 t/4}
\prod_{j=1}^n a_{k_j}^j  \, dt \\
&=
\frac{1}{4 (\pi \cD)^{n/2}}  \itg_0^\infty \e^{-   \l^2 t/4}
\e^{-|\bdk|^2/( \cD(1+ t))}\prod_{j=1}^n
\cP_M\big(\frac{t}{h_j^2 \cD},\frac{k_j}{\sqrt{\cD}}\big) \, dt \, .
\end{align*}

Following \cite{WALD} the one-dimensional integrals  of
$\bda_{\bdk}^{(M)}$ and $\bdb_{\bdk,\bdm}^{(M)}$ are transformed
to integrals over $\R$ with
integrands
decaying doubly exponentially
by making the substitutions
\begin{equation}\label{wald}
t=\e^\xi,\quad \xi=\alpha (\sigma+\e^\sigma),\quad \sigma=\beta (u-\e^{-u})
\end{equation}
with certain positive constants $\alpha,\beta$, and the computation is based on
the classical trapezoidal rule.
Then the tensor product structure of the integrands allows the efficient
computation of the coefficients $\bdb_{\bdk,\bdm}^{(M)}$ and
$\bda_{\bdk}^{(M)}$.
Moreover, the computation of  the convolutional sum is very
efficient
for integrands, which allow a separated representation, i.e.,
for given accuracy $\epsilon$
they can be represented as a sum of products of vectors in dimension~$1$
\begin{align*}
&f(h_1m_1,\ldots,h_mm_n) = \sum_{p =1}^R r_p \prod_{j=1}^{n}
f_j^{(p)}(h_jm_j) + \cO(\epsilon) \, .
\end{align*}

In \cite{LMS2}
we have described this approach
to the fast computation of high dimensional volume potentials
for compactly supported integrands.
To compute the convolutional sum
\begin{align*}
\sum_{\bdh\bdm \in\Omega_{r \bdh}}
\hspace{-5pt}\bda_{\bdk-\bdm}^{(M)} \, {f}(\bdh\bdm)
\end{align*}
we get after the substitutions
\begin{align*}
\bda^{(M)}_{\bdk} =
\frac{1}{4 (\pi \cD)^{n/2}}  \int_{-\infty}^\infty  \e^{-   \l^2 \Phi(u)/4}
\e^{-|\bdk|^2/( \cD(1+ \Phi(u))}\prod_{j=1}^n
\cP_M\big(\frac{\Phi(u)}{h_j^2 \cD},\frac{k_j}{\sqrt{\cD}}\big) 
\,\Phi'(u) \,  du \, ,
\end{align*}
where we set
\begin{align*}
\Phi(u)&=\exp(\alpha \beta (u-\exp(-u)) +\alpha \exp(\beta(u-\exp(-u)))),
\\
\Phi'(u)&=\Phi(u) \alpha \beta (1+\e^{-u})(1+\exp(\beta(u-\exp(-u)))).
\end{align*}
The quadrature with the  trapezoidal
rule with step size $\tau$
\begin{align*}
\bda_{\bdk}^{(M)}\approx \frac{\tau}{4(\pi \cD)^{n/2} } \sum_{s=-N_0}^{N_1}
\e^{- \l^2 \Phi(s\,\tau)/4}\e^{-|\bdk|^2/( \cD(1+   \Phi(s\,\tau)))}
\prod_{j=1}^n\cP_M\big(\frac{\Phi(s\,\tau)}{h_j^2 
\cD},\frac{k}{\sqrt{\cD}}\big)
\Phi'(s\,\tau)
\end{align*}
provides the approximation via one-dimensional discrete convolutions
\begin{align*}
\sum_{\bdh\bdm \in\Omega_{r \bdh}}
\bda_{\bdk-\bdm} \, f(\bdh\bdm)  \approx \frac{\tau}{4(\pi 
\cD)^{n/2} } \sum_{p=1}^R r_p
\sum_{s=-N_0}^{N_1} \e^{- \l^2 \Phi(s\,\tau)/4}
\Phi'(s\,\tau)\\
\times
\prod_{j=1}^{n}\sum_{m_j}
\e^{-(k_j-m_j)^2/( \cD(1+   \Phi(s\,\tau)))}
P_M\big(\frac{\Phi(s\,\tau)}{h_j^2 \cD},\frac{k_j -m_j}{\sqrt{\cD}}\big)
f_j^{(p)}(h_jm_j)\, .
\end{align*}

\bigskip

We provide some numerical tests to the approximation of the potential $\cK_\l f$ over the cube $[-1,1]^n$, $n\geq 3$, with the density 
\begin{equation}\label{product}
f(\bdx)=(-\Delta +\lambda^2) \prod_{j=1}^n u(x_j)=\sum_{p =1}^{n} \prod_{j=1}^{n}
f_j^{(p)}(x_j)  ,\> \bdx=(x_1,...,x_n)\in[-1,1]^n;
\end{equation}
\[
f_j^{(p)}(x)=u(x) \quad\hbox{ if} \quad j\neq p; \>f_j^{(p)}(x)=-u''(x)+\frac{\l^2}{n}u(x)\quad\hbox{ if}\quad \> j=p.
\]
Let $\widetilde{f}_j^{(p)}$ be an extension of $f_j^{(p)}$ outside the interval $[-1,1]$ with preserved smoothness and
\begin{equation*}\label{product2}
\widetilde{f}(\bdx)=\sum_{p =1}^{n} \prod_{j=1}^{n}
\widetilde{f}_j^{(p)}(x_j),\> \bdx\in \R^n.
\end{equation*}
By using  Hestenes reflection principle (\cite{He}) we construct an extension of  $f_j^{(p)}$ outside the interval $[-1,1]$ as 
\begin{equation*}\label{ext}
\widetilde{f}_j^{(p)}(x)=
\left\{
\begin{tabular}{cc}
$\ds\sum_{s=1}^{N+1} c_s f_j^{(p)}(-a_s\,(x+1)-1)$,& $x<-1$\\[1mm]
$f_j^{(p)}(x)$,& $-1\leq x \leq 1$\\[1mm]
$\ds\sum_{s=1}^{N+1} c_s f_j^{(p)}(-a_s\,(x-1)+1)$,& $x>1$\\[1mm]
$$
\end{tabular}
\right.
\end{equation*}
where $a_1,...,a_{N+1}$ are different  positive constants and the coefficients 
$\mathbf{c}_N=\{c_1,...,c_{N+1}\}$ are the unique solution of  the $(N+1)\times(N+1)$
system of linear equations
\begin{equation*}\label{sis}
\sum_{s=1}^{N+1} c_s (- a_s)^{k}=1,\quad k=0,...,N.
\end{equation*}
We provide results for  $\widetilde{f}_j^{(p)}=f_j^{(p)}$ and three different Hestenes extensions  corresponding to $a_s=2^{-s}$ (Extension 1) , $a_s=s^{-1}$ (Extension 2), $a_s=s$ (Extension 3).

 The approximation values are computed  by the cubature formula 
\eqref{defquadbox2}  for  $h_j=h$, $j=1,...,n$.
To have the saturation error comparable with the double precision rounding errors,
 we have chosen the parameter $\cD=4$.

In Tables \ref{table1}, \ref{table2}  and \ref{table3} we report on the absolute error and the approximation rate for the three-dimensional potential $\cK_\lambda f $,  when
$u(x)=\cos^2(\pi x/2)$ (Table  \ref{table1}),
$u(x)=(x^2-1)^3$ (Table  \ref{table2}) and
$u(x)=(x^2-1)^2$ (Table  \ref{table3}), in the case  $\lambda^2=1$ and $\lambda^2=1+i$.
 We have chosen the parameters $\a=2, \b=2$ in the transformations \eqref{wald} and
 $\tau=0.005$, $N_1=-N_0=300$ in the quadrature formula.
 The numerical results confirm the $h^2-$, $h^4-$ and, respectively, $h^6-$ convergence of the cubature formulas
 \eqref{defquadbox2} when $M=1,2,3$. For extensions 1, 2 and 3 the numerical results are similar with those if using $\widetilde{f}_j^{(p)}=f_j^{(p)}$. In Table \ref{table3} we see that the error of the approximate quasi-interpolant of order $6$ has reached the saturation bound. This is a feature of the method that approximate quasi-interpolant of order $N$ reproduces polynomials of degree $<N$ up to the saturation error.
 
To check the effectiveness of the method for very high dimension $n$ we computed the potential over $[-1,1]^n$ of the density \eqref{product} with $u(x)=1-\sin(\pi x^2/2)$ (Table \ref{table4}) and $u(x)=\e^x (1-x^2)^2$ (Table \ref{table5})  in dimension $n=10^i$, $i=1,..,8$ and different extensions.
We have chosen $a=6$, $b=5$, $\tau=0.003$, $N_0=-40$, $N_1=200$.
The results show that  $\cK_{\l,h}^{(3)}$  approximates with the predicted approximation rate $6$, also for very large $n$ and the error scales linearly in the space dimension.

\begin{table}[h]
\begin{scriptsize}
\begin{center}
\begin{tabular}{c|c|cc|cc|cc}
\multicolumn{1}{c}{$\lambda^2=1$:}&  \multicolumn{6}{c}{}\\
&& \multicolumn{2}{c|}{$M=1$}&\multicolumn{2}{c|}{$M=2$}&\multicolumn{2}{c}{$M=3$}\\
\hline
$\widetilde{f}(\bdx)$ &$h^{-1}$& error & rate&  error & rate & error & rate
\\
\hline
& 10      &0.822E-01    &           &0.414E-02    &          &  0.135E-03&          \\
& 20      & 0.219E-01 &   1.9062 & 0.272E-03 &   3.9267    & 0.223E-05 &   5.9201     \\
$f(\bdx)$& 40      & 0.557E-02 &   1.9760  &   0.172E-04 &   3.9821 &  0.354E-07 &   5.9800     \\
 & 80     &  0.140E-02 &   1.9940   & 0.108E-05 &   3.9955  & 0.555E-09 &   5.9950    \\
& 160     &    0.350E-03 &   1.9985 & 0.675E-07 &   3.9989 &   0.867E-11 &   5.9987     \\
& 320    & 0.875E-04 &   1.9996      & 0.422E-08 &   3.9997  & 0.136E-12 &   5.9994      \\
\hline
& 10      &  0.821E-01          &   & 0.413E-02         &  & 0.135E-03&        \\
& 20      &0.219E-01 &   1.9057    &0.272E-03 &   3.9265    &  0.223E-05 &   5.9201     \\
ext 1& 40      &   0.557E-02 &   1.9760 &0.172E-04 &   3.9820   &0.354E-07 &   5.9800    \\
 & 80     &  0.140E-02 &   1.9940&0.108E-05 &   3.9955   & 0.554E-09 &   5.9961    \\
& 160     & 0.350E-03 &   1.9985  &  0.675E-07 &   3.9989    & 0.825E-11 &   6.0692  \\
& 320    &0.875E-04 &   1.9996  & 0.422E-08 &   3.9997   & 0.789E-12 &   3.3868\\
\hline
& 10      &   0.826E-01 &           &  0.422E-02  &          &  0.140E-03&      \\
& 20      & 0.219E-01 &   1.9138     & 0.273E-03 &   3.9520    & 0.224E-05 &   5.9686 \\
ext 2& 40      & 0.557E-02 &   1.9769     & 0.172E-04 &   3.9850 &0.354E-07 &   5.9856  \\
 & 80     &  0.140E-02 &   1.9941  &0.108E-05 &   3.9959  &0.554E-09 &   5.9967 \\
& 160     & 0.350E-03 &   1.9985   &  0.675E-07 &   3.9989 & 0.883E-11 &   5.9718   \\
& 320    & 0.875E-04 &   1.9996     & 0.422E-08 &   3.9997&  0.120E-12 &   6.1971 \\
\hline
& 10      &   0.946E-01 &           & 0.139E-01   &          &  0.260E-01&       \\
& 20      &    0.224E-01 &   2.0769   & 0.771E-03 &   4.1771   &  0.871E-04 &   8.2194 \\
ext 3& 40      &  0.559E-02 &   2.0047    & 0.228E-04 &   5.0788 &0.111E-05 &   6.2957   \\
 & 80     & 0.140E-02 &   1.9977  & 0.113E-05 &   4.3396&0.341E-08 &   8.3438 \\
& 160     & 0.350E-03 &   1.9990 &  0.679E-07 &   4.0529   &0.147E-10 &   7.8633 \\
& 320    &0.875E-04 &   1.9997      &0.422E-08 &   4.0067 & 0.147E-12 &   6.6382\\
\hline
\end{tabular}
\medskip

\begin{tabular}{c|c|cc|cc|cc}
\multicolumn{1}{c}{$\lambda^2=1+i$:}&  \multicolumn{6}{c}{}\\
&& \multicolumn{2}{c|}{$M=1$}&\multicolumn{2}{c|}{$M=2$}&\multicolumn{2}{c}{$M=3$}\\
\hline
$\widetilde{f}(\bdx)$ &$h^{-1}$& error & rate&  error & rate & error & rate
\\
\hline
& 10      &  0.815E-01   &           &  0.410E-02  &          &0.134E-03   &          \\
& 20      & 0.217E-01 &   1.9060   & 0.270E-03 &   3.9267 &  0.221E-05 &   5.9201    \\
$f(\bdx)$& 40      & 0.553E-02 &   1.9760    &   0.171E-04 &   3.9821   & 0.351E-07 &   5.9800   \\
 & 80     & 0.139E-02 &   1.9940& 0.107E-05 &   3.9955       &  0.550E-09 &   5.9950 \\
& 160     & 0.347E-03 &   1.9985  & 0.669E-07 &   3.9989&  0.860E-11 &   5.9987  \\
& 320    &0.868E-04 &   1.9996 & 0.418E-08 &   3.9997 &0.135E-12 &   5.9974  \\
\hline
& 10      &    0.814E-01 &           &  0.410E-02  &          &0.134E-03  &        \\
& 20      & 0.217E-01 &   1.9055   & 0.270E-03 &   3.9265  &0.221E-05 &   5.9201\\
ext 1& 40      &  0.553E-02 &   1.9759   & 0.171E-04 &   3.9820   & 0.351E-07 &   5.9800    \\
 & 80     &  0.139E-02 &   1.9940&  0.107E-05 &   3.9955    &  0.550E-09 &   5.9959  \\
& 160     &   0.347E-03 &   1.9985 & 0.669E-07 &   3.9989  &   0.826E-11 &   6.0555   \\
& 320    & 0.868E-04 &   1.9996    &    0.418E-08 &   3.9997 &  0.710E-12 &   3.5419 \\
\hline
& 10      &  0.819E-01   &           & 0.417E-02   &          & 0.139E-03 &      \\
& 20      & 0.218E-01 &   1.9127 &  0.270E-03 &   3.9490   &  0.222E-05 &   5.9631   \\
ext 2& 40      &   0.553E-02 &   1.9768& 0.171E-04 &   3.9846  & 0.351E-07 &   5.9849   \\
 & 80     &0.139E-02 &   1.9941 &   0.107E-05 &   3.9959  & 0.550E-09 &   5.9964\\
& 160     &  0.347E-03 &   1.9985  & 0.669E-07 &   3.9989 & 0.873E-11 &   5.9767   \\
& 320    & 0.868E-04 &   1.9996 &  0.418E-08 &   3.9997 &  0.122E-12 &   6.1594 \\
\hline
& 10      &  0.924E-01  &           &  0.130E-01 &          & 0.238E-01  &       \\
& 20      &  0.222E-01 &   2.0586  &  0.717E-03 &   4.1823  &   0.799E-04 &   8.2188  \\
ext 3& 40      &   0.554E-02 &   2.0011   &    0.220E-04 &   5.0283 &  0.101E-05 &   6.3010   \\
 & 80     &   0.139E-02 &   1.9973  & 0.111E-05 &   4.3051   &  0.313E-08 &   8.3370 \\
& 160     &  0.347E-03 &   1.9989 &  0.673E-07 &   4.0461  &0.139E-10 &   7.8181 \\
& 320    &0.868E-04 &   1.9997  &  0.419E-08 &   4.0058 &0.145E-12 &   6.5840   \\
\hline
\end{tabular}
\end{center}
\caption{Absolute errors and approximation rates
for $\cK_\l f(0.3,0.3,0)$ using $\cK_{\l,h}^{(M)} f(0.3,0.3,0)$ with the density $f$ given in \eqref{product} with $u(x)=\cos^2(\pi x/2)$ and  different extensions, $M=1,2,3$, $\lambda^2=1$ and $\lambda^2=1+i$.}
\label{table1}
\end{scriptsize}
\end{table}

\begin{table}[h]
\begin{scriptsize}
\begin{center}
\begin{tabular}{c|c|cc|cc|cc}
\multicolumn{1}{c}{$\lambda^2=1$:}&  \multicolumn{7}{c}{}\\
&& \multicolumn{2}{c|}{$M=1$}&\multicolumn{2}{c|}{$M=2$}&\multicolumn{2}{c}{$M=3$}\\
\hline
$\widetilde{f}(\bdx)$ &$h^{-1}$& error & rate&  error & rate & error & rate\\
\hline
& 10      & 0.673E-01   &           &0.626E-02   &          &  0.427E-04&     \\
& 20      & 0.159E-01 &   2.0819     &0.392E-03 &   3.9965  &  0.668E-06 &   5.9997    \\
$f(\bdx)$& 40      & 0.391E-02 &   2.0238    &0.246E-04 &   3.9970&  0.104E-07 &   6.0000 \\
 & 80     &   0.973E-03 &   2.0062   & 0.154E-05 &   3.9991 & 0.163E-09 &   6.0000   \\
 & 160     &  0.243E-03 &   2.0016      & 0.960E-07 &   3.9998&   0.255E-11 &   6.0000  \\
 & 320    &  0.607E-04 &   2.0004     & 0.600E-08 &   3.9999&   0.398E-13 &   6.0002    \\
\hline
& 10      &0.637E-01    &           &0.634E-02   &          & 0.427E-04 &      \\
& 20      &  0.157E-01 &   2.0254  & 0.393E-03 &   4.0094  &  0.668E-06 &   5.9997    \\
ext 1& 40      & 0.389E-02 &   2.0075  & 0.246E-04 &   4.0003& 0.104E-07 &   5.9995   \\
 & 80     &0.972E-03 &   2.0020   & 0.154E-05 &   3.9999 &  0.156E-09 &   6.0635  \\
 & 160     & 0.243E-03 &   2.0005      &  0.961E-07 &   4.0000  &   0.389E-11 &   5.3255    \\
 & 320    & 0.607E-04 &   2.0001  &   0.600E-08 &   4.0000  &   0.603E-12 &   2.6899   \\
 \hline
& 10      &0.603E-01 &           &  0.644E-02 &          &  0.427E-04&          \\
& 20      &0.154E-01 &   1.9662 &0.395E-03 &   4.0264 &  0.668E-06 &   5.9997   \\
ext 2& 40      & 0.388E-02 &   1.9925 & 0.246E-04 &   4.0052 &  0.104E-07 &   6.0003  \\
 & 80     & 0.971E-03 &   1.9983 & 0.154E-05 &   4.0012 &  0.163E-09 &   6.0019\\
 & 160     & 0.243E-03 &   1.9996 & 0.961E-07 &   4.0003 &   0.224E-11 &   6.1838 \\
 & 320    &0.607E-04 &   1.9999   & 0.600E-08 &   4.0001&   0.408E-12 &   2.4557    \\
\hline
& 10      &0.291E-01    &       &0.626E-02  &       &  0.427E-04      &    \\
& 20      & 0.133E-01 &   1.1335  & 0.392E-03 &   3.9965  &  0.668E-06 &   5.9997  \\
ext 3& 40      &0.374E-02 &   1.8264 & 0.246E-04 &   3.9970& 0.104E-07 &   6.0000    \\
 & 80     & 0.963E-03 &   1.9586   & 0.154E-05 &   3.9991  & 0.163E-09 &   6.0000   \\
 & 160     &  0.224E-03 &   1.9894    & 0.960E-07 &   3.9998  &   0.255E-11 &   6.0000   \\
 & 320    & 0.607E-04 &   1.9975 & 0.600E-08 &   3.9999  &  0.398E-13 &   6.0001    \\
 \hline
\end{tabular}
\medskip

\begin{tabular}{c|c|cc|cc|cc}
\multicolumn{1}{c}{$\lambda^2=1+i$:}& \multicolumn{7}{c}{}\\
&& \multicolumn{2}{c|}{$M=1$}&\multicolumn{2}{c|}{$M=2$}&\multicolumn{2}{c}{$M=3$}\\
\hline
$\widetilde{f}(\bdx)$ &$h^{-1}$& error & rate&  error & rate & error & rate\\
\hline
& 10     &0.604E-01    &           &0.572E-02   &          &  0.441E-04 &       \\
& 20    &    0.142E-01 &   2.0834          &0.358E-03 &   3.9963         & 0.690E-06 &   5.9997 \\
$f(\bdx)$& 40 &    0.350E-02 &   2.0242             &    0.224E-04 &   3.9969         &  0.108E-07 &   6.0000    \\
 & 80  &  0.872E-03 &   2.0062        &  0.140E-05 &   3.9991  &  0.168E-09 &   6.0000  \\
& 160  &  0.218E-03 &   2.0016        &0.878E-07 &   3.9998       &  0.263E-11 &   6.0000  \\
& 320   &  0.544E-04 &   2.0004         &0.548E-08 &   3.9999    &  0.410E-13 &   6.0025  \\
\hline
& 10   &0.572E-01    &           &0.579E-02 &          & 0.441E-04 &      \\
& 20  &0.140E-01 &   2.0271&  0.360E-03 &   4.0096   &   0.690E-06 &   5.9997  \\
ext 1& 40       &    0.349E-02 &   2.0080          & 0.225E-04 &   4.0004&  0.108E-07 &   5.9996   \\
 & 80     & 0.871E-03 &   2.0021         & 0.140E-05 &   4.0000 &  0.163E-09 &   6.0465  \\
& 160       &   0.218E-03 &   2.0006          &    0.878E-07 &   4.0000&  0.372E-11 &   5.4561 \\
& 320    &  0.544E-04 &   2.0001           & 0.548E-08 &   4.0000 &  0.539E-12 &   2.7853    \\
\hline
& 10         &0.542E-01    &           &0.589E-02   &          & 0.441E-04 &  \\
& 20       &   0.138E-01 &   1.9681           & 0.361E-03 &   4.0272          &  0.690E-06 &   5.9997   \\
ext 2& 40     &  0.348E-02 &   1.9931          & 0.225E-04 &   4.0055 &  0.108E-07 &   6.0002   \\
 & 80      &  0.870E-03 &   1.9984          & 0.140E-05 &   4.0013 & 0.168E-09 &   6.0014  \\
& 160     &  0.218E-03 &   1.9996          &  0.878E-07 &   4.0003 &  0.240E-11 &   6.1310   \\
& 320    &  0.544E-04 &   1.9999          & 0.548E-08 &   4.0001 &  0.365E-12 &   2.7174   \\
\hline
& 10   & 0.261E-01           &   &    0.803E-02 & &  0.441E-04&      \\
& 20   &     0.119E-01 &   1.1338       &0.560E-03 &   3.8421  &  0.690E-06 &   5.9997   \\
ext 3& 40   &0.335E-02 &   1.8275  &  0.268E-04 &   4.3875&  0.108E-07 &   6.0000  \\
 & 80  &0.862E-03 &   1.9590   &    0.148E-05 &   4.1767& 0.168E-09 &   6.0000   \\
& 160   &    0.217E-03 &   1.9899          & 0.890E-07 &   4.0553 &  0.263E-11 &   6.0000 \\
& 320  &     0.544E-04 &   1.9975   & 0.550E-08 &   4.0151     &  0.410E-13 &   6.0030   \\
\hline
\end{tabular}
\end{center}
\caption{Absolute errors and approximation rates
for $\cK_\l f(0.5,0.5,0.5)$ using $\cK_{\l,h}^{(M)} f(0.5,0.5,0.5)$ with the density $f$ given in \eqref{product} with  $u(x)=(x^2-1)^3$ and  different extensions,  $M=1,2,3$, $\lambda^2=1$ and $\lambda^2=1+i$.}
\label{table2}
\end{scriptsize}
\end{table}

\begin{table}[h]
\begin{scriptsize}
\begin{center}
\begin{tabular}{c|c|cc|cc|cc}
\multicolumn{1}{c}{$\lambda^2=1$:}&  \multicolumn{6}{c}{}\\
&& \multicolumn{2}{c|}{$M=1$}&\multicolumn{2}{c|}{$M=2$}&\multicolumn{2}{c}{$M=3$}\\
\hline
$\widetilde{f}(\bdx)$ &$h^{-1}$& error & rate&  error & rate & error & rate
\\
\hline
& 10      &    0.935E-01&           &  0.166E-02 &          &0.222E-15  &          \\
& 20      &   0.241E-01 &   1.9564     & 0.104E-03 &   3.9984    &  0.777E-15&     \\
$f(\bdx)$& 40      &   0.607E-02 &   1.9883     &0.647E-05 &   3.9999 &0.111E-15  &     \\
 & 80     & 0.152E-02 &   1.9970  & 0.405E-06 &   4.0000  &  0.555E-16&     \\
& 160     &    0.380E-03 &   1.9993  & 0.253E-07 &   4.0000  &  0.555E-16 &     \\
& 320    & 0.951E-04 &   1.9998&  0.158E-08 &   4.0000  &  0.222E-15 &      \\
\hline
& 10      &  0.941E-01   &           & 0.166E-02  &          & 0.779E-10 &        \\
& 20      & 0.241E-01 &   1.9632   & 0.104E-03 &   3.9984  &  0.336E-10 &   1.2133 \\
ext 1& 40      &  0.607E-02 &   1.9903 & 0.647E-05 &   3.9999  & 0.143E-10 &   1.2318   \\
 & 80     & 0.152E-02 &   1.9975& 0.405E-06 &   4.0000&    0.628E-11 &   1.1873  \\
& 160     &0.380E-03 &   1.9994  & 0.253E-07 &   4.0000   &  0.160E-12 &   5.2966  \\
& 320    & 0.951E-04 &   1.9998&  0.158E-08 &   4.0000  &  0.268E-12 &  -0.7471  \\
\hline
& 10      &   0.946E-01   &           & 0.166E-02   &          & 0.201E-11 &      \\
& 20      &  0.242E-01 &   1.9684    & 0.104E-03 &   3.9984  &  0.133E-11 &   0.5895\\
ext 2& 40      & 0.607E-02 &   1.9920  & 0.647E-05 &   3.9999   & 0.139E-11 &  -0.0557   \\
 & 80     &  0.152E-02 &   1.9980 &  0.405E-06 &   4.0000   & 0.641E-13 &   4.4336  \\
& 160     & 0.380E-03 &   1.9995 &0.253E-07 &   4.0000& 0.532E-12 &  -3.0524 \\
& 320    & 0.951E-04 &   1.9999&   0.158E-08 &   4.0000 & 0.404E-12 &   0.3980 \\
\hline
& 10      & 0.983E-01   &           & 0.166E-02 &          &  0.222E-15 &       \\
& 20      & 0.245E-01 &   2.0041  & 0.104E-03 &   3.9984     & 0.722E-15 &   \\
ext 3& 40      & 0.610E-02 &   2.0066& 0.647E-05 &   3.9999  &  0.111E-15 &    \\
 & 80     &0.152E-02 &   2.0022&   0.405E-06 &   4.0000& 0.555E-16 & \\
& 160     & 0.380E-03 &   2.0006&0.253E-07 &   4.0000  &0.555E-16   &  \\
& 320    & 0.951E-04 &   2.0002&    0.158E-08 &   4.0000   &  0.111E-15 &     \\
\hline
\end{tabular}
\medskip

\begin{tabular}{c|c|cc|cc|cc}
\multicolumn{1}{c}{$\lambda^2=1+i$:}&  \multicolumn{6}{c}{}\\
&& \multicolumn{2}{c|}{$M=1$}&\multicolumn{2}{c|}{$M=2$}&\multicolumn{2}{c}{$M=3$}\\
\hline
$\widetilde{f}(\bdx)$ &$h^{-1}$& error & rate&  error & rate & error & rate
\\
\hline
& 10      &  0.869E-01  &           & 0.168E-02   &          & 0.220E-14  &          \\
& 20      & 0.224E-01 &   1.9541     & 0.105E-03 &   3.9983    &0.729E-15  &     \\
$f(\bdx)$& 40      & 0.565E-02 &   1.9878     &   0.655E-05 &   3.9999  &0.397E-15  &     \\
 & 80     &     0.142E-02 &   1.9969   & 0.410E-06 &   4.0000 &  0.555E-16&     \\
& 160     &  0.354E-03 &   1.9992     &  0.256E-07 &   4.0000  &  0.128E-15 &     \\
& 320    &  0.886E-04 &   1.9998 &    0.160E-08 &   4.0000    &0.906E-16   &      \\
\hline
& 10      &  0.875E-01  &           &  0.168E-02  &          & 0.695E-10 &        \\
& 20      &  0.225E-01 &   1.9617    &    0.105E-03 &   3.9983 &   0.301E-10 &   1.2058     \\
ext 1& 40      &0.566E-02 &   1.9899     &  0.655E-05 &   3.9999 & 0.128E-10 &   1.2359   \\
 & 80     &  0.142E-02 &   1.9974&  0.410E-06 &   4.0000 &  0.563E-11 &   1.1851 \\
& 160     &  0.354E-03 &   1.9994 & 0.256E-07 &   4.0000  &  0.144E-12 &   5.2918 \\
& 320    &  0.886E-04 &   1.9998   &   0.160E-08 &   4.0000& 0.240E-12 &  -0.7425\\
\hline
& 10      &   0.880E-01   &           &0.168E-02   &          & 0.179E-11 &      \\
& 20      &    0.225E-01 &   1.9675  & 0.105E-03 &   3.9983  &  0.119E-11 &   0.5930   \\
ext 2& 40      &  0.566E-02 &   1.9917 &  0.655E-05 &   3.9999 &  0.124E-11 &  -0.0642\\
 & 80     &   0.142E-02 &   1.9979 &  0.410E-06 &   4.0000&0.577E-13 &   4.4276 \\
& 160     & 0.354E-03 &   1.9995  &    0.256E-07 &   4.0000  & 0.476E-12 &  -3.0451 \\
& 320    &0.886E-04 &   1.9999   &   0.160E-08 &   4.0000  & 0.361E-12 &   0.3971  \\
\hline
& 10      &   0.918E-01 &           &0.168E-02   &          & 0.215E-14 &       \\
& 20      &    0.228E-01 &   2.0074& 0.105E-03 &   3.9983 &   0.625E-15 &     \\
ext 3& 40      &  0.568E-02 &   2.0074     &  0.655E-05 &   3.9999& 0.296E-15  &    \\
 & 80     & 0.142E-02 &   2.0024  &  0.410E-06 &   4.0000&  0.706E-17 &   \\
& 160     &   0.354E-03 &   2.0006   & 0.256E-07 &   4.0000  & 0.794E-16 &   \\
& 320    &  0.886E-04 &   2.0002  &  0.160E-08 &   4.0000  &   0.119E-15 &     \\
\hline
\end{tabular}
\end{center}
\caption{Absolute errors and approximation rates
for $\cK_\l f(0.4,0.5,0)$ using $\cK_{\l,h}^{(M)} f(0.4,0.5,0)$ with the density $f$ given in \eqref{product} with $u(x)=(1-x^2)^2$ and  different extensions,  with $M=1,2,3$,  $\lambda^2=1$ and $\lambda^2=1+i$.}
\label{table3}
\end{scriptsize}
\end{table}

 \begin{table}[h]
\begin{scriptsize}
\begin{center}
\begin{tabular}{c|r|cc|cc|cc|cc} \hline
         $\widetilde{f}(\bdx)$&    $n$& \multicolumn{2}{c|}{10}&\multicolumn{2}{c|}{$10^2$}& \multicolumn{2}{c|}{$10^3$}& \multicolumn{2}{c}{$10^4$}  \\ \hline
&$h^{-1}$&  error  & rate         &         error     &         rate      &     error       &       rate &     error       &       rate     \\ \hline
&$10$& 0.338E-03&  & 0.459E-02&&0.487E-01&&0.703E{\tiny +}00& \\
&$20  $&   0.605E-05 &   5.8020 &  0.732E-04 &   5.9727   &0.746E-03 &   6.0282  & 0.751E-02 &   6.5491\\
&$40  $&  0.976E-07 &   5.9541   & 0.115E-05 &   5.9966    &  0.117E-04 &   5.9991&0.117E-03 &   6.0070 \\
$f(\bdx)$&$80  $& 0.154E-08 &   5.9887 & 0.179E-07 &   5.9994    & 0.182E-06 &   5.9999 &0.183E-05 &   6.0000 \\
&$160  $& 0.241E-10 &   5.9971  & 0.280E-09 &   6.0013 &0.285E-08 &   6.0000  & 0.285E-07 &   5.9999 \\
&$320 $&   0.376E-12 &   5.9982  &  0.513E-11 &   5.7677  &  0.445E-10 &   6.0005  & 0.446E-09 &   5.9985  \\
\hline
\end{tabular}\\[0.5mm]

\begin{tabular}{c|r|cc|cc|cc|cc} \hline
  $\widetilde{f}(\bdx)$&           $n$& \multicolumn{2}{c|}{$10^5$}&\multicolumn{2}{c|}{$10^6$}& \multicolumn{2}{c|}{$10^7$}& \multicolumn{2}{c}{$10^8$}  \\ \hline
&$h^{-1}$&  error  & rate         &         error     &         rate      &     error       &       rate &     error       &       rate     \\ \hline
&$20  $&0.794E-01&  &0.145E{\tiny +}01  &  &  &  &\\
&$40  $&  0.117E-02 &   6.0852  & 0.118E-01 &   6.9443  & 0.129E{\tiny +}00 &  &0.348E{\tiny +}01  & \\
$f(\bdx)$&$80  $&   0.183E-04 &   6.0012&  0.183E-03 &   6.0133 & 0.183E-02 &   6.1364  &0.185E-01 &   7.5527 \\
&$160  $&   0.285E-06 &   5.9992  &  0.286E-05 &   5.9975 &  0.286E-04 &   5.9985  & 0.286E-03 &   6.0174 \\
&$320 $& 0.451E-08 &   5.9842  &0.478E-07 &   5.9030  & 0.510E-06 &   5.8096  & 0.517E-05 &   5.7889 \\
\hline
\end{tabular}\\[5mm]

\begin{tabular}{c|r|cc|cc|cc|cc} \hline
         $\widetilde{f}(\bdx)$&    $n$& \multicolumn{2}{c|}{10}&\multicolumn{2}{c|}{$10^2$}& \multicolumn{2}{c|}{$10^3$}& \multicolumn{2}{c}{$10^4$}  \\ \hline
&$h^{-1}$&  error  & rate         &         error     &         rate      &     error       &       rate &     error       &       rate     \\ \hline
&$10$&     0.352E-03& &0.459E-02&& 0.487E-01&&0.703E{\tiny +}00                           \\
&$20  $&  0.611E-05 &   5.8476  & 0.732E-04 &   5.9726  &0.746E-03 &   6.0282  & 0.751E-02 &   6.5491 \\
ext&$40  $&  0.978E-07 &   5.9652  &  0.115E-05 &   5.9966&0.117E-04 &   5.9991  & 0.117E-03 &   6.0070 \\
$1$&$80  $&  0.154E-08 &   5.9892  &  0.179E-07 &   5.9994 & 0.182E-06 &   5.9999  & 0.183E-05 &   6.0000 \\
&$160  $&  0.230E-10 &   6.0635  & 0.280E-09 &   6.0013  & 0.285E-08 &   6.0000  & 0.285E-07 &   5.9999 \\
&$320 $&   0.650E-12 &   5.1472  &  0.513E-11 &  5.7677  & 0.445E-10 &   6.0005  &  0.446E-09 &   5.9985 \\
\hline
\end{tabular}\\[0.5mm]

\begin{tabular}{c|r|cc|cc|cc|cc} \hline
  $\widetilde{f}(\bdx)$&           $n$& \multicolumn{2}{c|}{$10^5$}&\multicolumn{2}{c|}{$10^6$}& \multicolumn{2}{c|}{$10^7$}& \multicolumn{2}{c}{$10^8$}  \\ \hline
&$h^{-1}$&  error  & rate         &         error     &         rate      &     error       &       rate &     error       &       rate     \\ \hline
&$20$&     0.794E-01&&  0.145E{\tiny +}01&&            &&                   \\
ext&$40  $& 0.117E-02 &   6.0852  & 0.118E-01 &   6.9443 & 0.129E{\tiny +}00  &  &0.348E{\tiny +}01  & \\
$1$&$80  $& 0.183E-04 &   6.0012  & 0.183E-03 &   6.0133  & 0.183E-02 &   6.1364& 0.185E-01 &   7.5527 \\
&$160  $& 0.285E-06 &   5.9992 &  0.286E-05 &   5.9975  & 0.286E-04 &   5.9985 &0.286E-03 &   6.0174 \\
&$320 $& 0.451E-08 &   5.9842  & 0.478E-07 &   5.9030  &0.510E-06 &   5.8096 & 0.517E-05 &   5.7889 \\
\hline
\end{tabular}\\[5mm]

\begin{tabular}{c|r|cc|cc|cc|cc} \hline
         $\widetilde{f}(\bdx)$&    $n$& \multicolumn{2}{c|}{10}&\multicolumn{2}{c|}{$10^2$}& \multicolumn{2}{c|}{$10^3$}& \multicolumn{2}{c}{$10^4$}  \\ \hline
&$h^{-1}$&  error  & rate         &         error     &         rate      &     error       &       rate &     error       &       rate     \\ \hline
&$10$& 0.415E-03&&0.459E-02&& 0.487E-01&& 0.703E{\tiny +}00\\
&$20  $& 0.632E-05 &   6.0374  &  0.732E-04 &   5.9727   & 0.746E-03 &   6.0282&  0.751E-02 &   6.5491 \\
ext&$40  $&  0.985E-07 &   6.0037   &  0.115E-05 &   5.9966  &  0.117E-04 &   5.9991 &0.117E-03 &   6.0070  \\
$2$&$80  $&0.154E-08 &   5.9994 &   0.179E-07 &   5.9994  &   0.182E-06 &   5.9999   & 0.183E-05 &   6.0000 \\
&$160  $& 0.241E-10 &   5.9999  &  0.280E-09 &   6.0013  &0.285E-08 &   6.0000&0.285E-07 &   5.9999  \\
&$320 $&  0.408E-12 &   5.8832  & 0.513E-11 & 5.7677  &   0.445E-10 &   6.0005 &0.446E-09 &   5.9985  \\
\hline
\end{tabular}\\[0.5mm]

\begin{tabular}{c|r|cc|cc|cc|cc} \hline
  $\widetilde{f}(\bdx)$&           $n$& \multicolumn{2}{c|}{$10^5$}&\multicolumn{2}{c|}{$10^6$}& \multicolumn{2}{c|}{$10^7$}& \multicolumn{2}{c}{$10^8$}  \\ \hline
&$h^{-1}$&  error  & rate         &         error     &         rate      &     error       &       rate &     error       &       rate     \\ \hline
&$20$& 0.794E-01&&0.145E{\tiny +}01 && &&\\
ext&$40  $&   0.117E-02 &   6.0852  &   0.118E-01 &   6.9443  &0.129E{\tiny +}00  &  & 0.348E{\tiny +}01  & \\
$2$&$80  $& 0.183E-04 &   6.0012   &   0.183E-03 &   6.0133 &  0.183E-02 &   6.1364  & 0.185E-01 &   7.5527 \\
&$160  $& 0.285E-06 &   5.9992  & 0.286E-05 &   5.9975  & 0.286E-04 &   5.9985& 0.286E-03 &   6.0174 \\
&$320 $&  0.451E-08 &   5.9842   & 0.478E-07 &   5.9030  & 0.510E-06 &   5.8096  & 0.517E-05 &   5.7889 \\
\hline
\end{tabular}\\[2mm]

\caption{ Absolute errors and approximation rates
for $\cK_\lambda f(0.5,0...,0)$ using $\cK_{\lambda,h}^{(3)} f(0.5,0,...,0)$ with the density $f$ given in \eqref{product} with $u(x)=1-\sin(\pi x^2/2)$ and different extensions , $n=10^i$, $i=1,...,8$, $\lambda^2=1$.}\label{table4}
\end{center}
\end{scriptsize}
\end{table}

 \begin{table}[h]
\begin{scriptsize}
\begin{center}
\begin{tabular}{c|r|cc|cc|cc|cc} \hline
         $\widetilde{f}(\bdx)$&    $n$& \multicolumn{2}{c|}{10}&\multicolumn{2}{c|}{$10^2$}& \multicolumn{2}{c|}{$10^3$}& \multicolumn{2}{c}{$10^4$}  \\ \hline
&$h^{-1}$&  error  & rate         &         error     &         rate      &     error       &       rate &     error       &       rate     \\ \hline
&$10  $&0.699E-03  &  &0.596E-02  &  & 0.595E-01 &  & 0.759E{\tiny +}00  & \\
&$20  $& 0.106E-04 &   6.0400 &  0.902E-04 &   6.0453   &  0.880E-03 &   6.0792  & 0.881E-02 &   6.4288 \\
&$40  $& 0.165E-06 &   6.0100 &0.140E-05 &   6.0105 &0.136E-04 &   6.0111 &0.136E-03 &   6.0162 \\
$f(\bdx)$&$80  $&0.257E-08 &   6.0025 & 0.218E-07 &   6.0026&0.213E-06 &   6.0026 & 0.212E-05 &   6.0027\\
&$160  $& 0.402E-10 &   6.0005&0.341E-09 &   6.0017& 0.332E-08 &   6.0006  &  0.332E-07 &   6.0005 \\
&$320 $&0.632E-12 &   5.9909& 0.491E-11 &   6.1156& 0.585E-10 &   5.9998  &0.519E-09 &   5.9973\\
\hline
\end{tabular}\\[0.5mm]

\begin{tabular}{c|r|cc|cc|cc|cc} \hline
  $\widetilde{f}(\bdx)$&           $n$& \multicolumn{2}{c|}{$10^5$}&\multicolumn{2}{c|}{$10^6$}& \multicolumn{2}{c|}{$10^7$}& \multicolumn{2}{c}{$10^8$}  \\ \hline
&$h^{-1}$&  error  & rate         &         error     &         rate      &     error       &       rate &     error       &       rate     \\ \hline
&$20  $& 0.913E-01 &    &0.134E{\tiny +}01  &  &  &  &  & \\
&$40  $&    0.136E-02 &   6.0671  & 0.137E-01 &   6.6101 &0.145E{\tiny +}00 &    &0.267E{\tiny +}01 \\
$f(\bdx)$&$80  $&   0.212E-04 &   6.0035&0.212E-03 &   6.0113 &  0.212E-02 &   6.0906  &  0.214E-01 &   6.9639 \\
&$160  $&  0.332E-06 &   5.9994  & 0.332E-05 &   5.9966&   0.333E-04 &   5.9966& 0.333E-03 &   6.0087 \\
&$320 $&0.526E-08 &   5.9779&  0.572E-07 &   5.8594  & 0.632E-06 &   5.7186  &  0.646E-05 &   5.6865 \\
\hline
\end{tabular}\\[5mm]

\begin{tabular}{c|r|cc|cc|cc|cc} \hline
         $\widetilde{f}(\bdx)$&    $n$& \multicolumn{2}{c|}{10}&\multicolumn{2}{c|}{$10^2$}& \multicolumn{2}{c|}{$10^3$}& \multicolumn{2}{c}{$10^4$}  \\ \hline
&$h^{-1}$&  error  & rate         &         error     &         rate      &     error       &       rate &     error       &       rate     \\ \hline
&$10  $& 0.690E-03 &  &0.596E-02  &  & 0.595E-01   &  &0.759E{\tiny +}00  & \\
&$20  $&0.106E-04 &   6.0254 & 0.902E-04 &   6.0453  & 0.880E-03 &   6.0792  &  0.881E-02 &   6.4288 \\
ext&$40  $& 0.165E-06 &   6.0068  &  0.140E-05 &   6.0105  &  0.136E-04 &   6.0111  &  0.136E-03 &   6.0162 \\
$2$&$80  $&  0.257E-08 &   6.0019  &  0.218E-07 &   6.0026  & 0.213E-06 &   6.0026  & 0.212E-05 &   6.0027 \\
&$160  $& 0.401E-10 &   6.0046  & 0.341E-09 &   6.0017  & 0.332E-08 &   6.0006  & 0.332E-07 &   6.0005 \\
&$320 $& 0.676E-12 &   5.8884 & 0.491E-11 &   6.1156  &0.519E-10 & 5.9998 &0.519E-09 &   5.9973  \\
\hline
\end{tabular}\\[0.5mm]

\begin{tabular}{c|r|cc|cc|cc|cc} \hline
  $\widetilde{f}(\bdx)$&           $n$& \multicolumn{2}{c|}{$10^5$}&\multicolumn{2}{c|}{$10^6$}& \multicolumn{2}{c|}{$10^7$}& \multicolumn{2}{c}{$10^8$}  \\ \hline
&$h^{-1}$&  error  & rate         &         error     &         rate      &     error       &       rate &     error       &       rate     \\ \hline
&$20  $&  0.913E-01&  & 0.134E{\tiny +}01   &  &  &  &  & \\
ext&$40  $&  0.136E-02 &   6.0671  & 0.137E-01 &   6.6101 &  0.145E{\tiny +}00  &&0.267E{\tiny +}01 \\
$2$&$80  $&0.212E-04 &   6.0035&  0.212E-03 &   6.0113 &0.212E-02 &   6.0906  &0.214E-01 &   6.9639 \\
&$160  $& 0.332E-06 &   5.9994&   0.332E-05 &   5.9966  &  0.333E-04 &   5.9966  &  0.333E-03 &   6.0087 \\
&$320 $&  0.526E-08 &   5.9779  &  0.572E-07 &   5.8594   &  0.632E-06 &   5.7186  &0.646E-05 &   5.6865 \\
\hline
\end{tabular}\\[5mm]

\begin{tabular}{c|r|cc|cc|cc|cc} \hline
         $\widetilde{f}(\bdx)$&    $n$& \multicolumn{2}{c|}{10}&\multicolumn{2}{c|}{$10^2$}& \multicolumn{2}{c|}{$10^3$}& \multicolumn{2}{c}{$10^4$}  \\ \hline
&$h^{-1}$&  error  & rate         &         error     &         rate      &     error       &       rate &     error       &       rate     \\ \hline
&$10  $&   0.156E-01 &  &0.590E-02  &  &0.595E-01  &  & 0.759E{\tiny +}00 & \\
&$20  $& 0.165E-04 &   9.8811  &  0.901E-04 &   6.0349  &   0.880E-03 &   6.0791  &  0.881E-02 &   6.4288 \\
ext&$40  $&0.943E-07 &   7.4538  &  0.140E-05 &   6.0091  &  0.136E-04 &   6.0111  &0.136E-03 &   6.0162 \\
$3$&$80  $& 0.110E-08 &   6.4188  & 0.218E-07 &   6.0021 &0.213E-06 &   6.0026  &0.212E-05 &   6.0027 \\
&$160  $&0.333E-10 &   5.0496  & 0.340E-09 &   6.0016  &0.332E-08 &   6.0006 & 0.332E-07 &   6.0005 \\
&$320 $&0.602E-12 &   5.7901  &0.491E-11 &   6.1156& 0.519E-10 &5.9998&0.519E-09 &   5.9973 \\
\hline
\end{tabular}\\[0.5mm]

\begin{tabular}{c|r|cc|cc|cc|cc} \hline
  $\widetilde{f}(\bdx)$&           $n$& \multicolumn{2}{c|}{$10^5$}&\multicolumn{2}{c|}{$10^6$}& \multicolumn{2}{c|}{$10^7$}& \multicolumn{2}{c}{$10^8$}  \\ \hline
&$h^{-1}$&  error  & rate         &         error     &         rate      &     error       &       rate &     error       &       rate     \\ \hline
&$20  $&0.913E-01  &  &  0.134E{\tiny +}01  &  &  &  &  & \\
ext&$40  $&   0.136E-02 &   6.0671  &  0.137E-01 &   6.6101   &0.145E{\tiny +}00  &  &0.267E{\tiny +}01 \\
$3$&$80  $& 0.212E-04 &   6.0035 &  0.212E-03 &   6.0113  &  0.212E-02 &   6.0906  &  0.214E-01 &   6.9639 \\
&$160  $& 0.332E-06 &   5.9994  & 0.332E-05 &   5.9966  &0.333E-04 &   5.9966 &0.333E-03 &   6.0087 \\
&$320 $&   0.526E-08 &   5.9779  & 0.572E-07 &   5.8599  & 0.632E-06 &   5.7186  &  0.646E-05 &   5.6865 \\
\hline
\end{tabular}\\[2mm]

\caption{ Absolute errors and approximation rates
for $\cK_\lambda f(0.4,0.4,0,...,0)$ using $\cK_{\lambda,h}^{(3)} f(0.4,0.4,0,...,0)$ with the density $f$ given in \eqref{product} with $u(x)=\e^x (1-x^2)^2$ and different extensions , $n=10^i$, $i=1,...,8$, $\lambda^2=1$.}\label{table5}
\end{center}
\end{scriptsize}
\end{table}

\end{document}